\documentclass[11pt,English]{amsart}

\usepackage{amsmath} %
\usepackage{graphicx}

\usepackage{amsthm}
\usepackage{amsfonts}
\usepackage{amssymb}
\usepackage{mathtools}
\usepackage[hidelinks]{hyperref}
\usepackage[capitalize]{cleveref}

\usepackage{bookmark}
\usepackage[numbers,square]{natbib}
\usepackage{tikz}
\usetikzlibrary{calc,arrows.meta,bending,positioning,decorations.pathreplacing}

\usepackage{ifthen}

\DeclareMathAlphabet{\mymathbb}{U}{BOONDOX-ds}{m}{n}

\crefname{sec}{Section}{Sections}

\newcommand{\Rb}{\mathbb{R}}
\newcommand{\Nb}{\mathbb{N}}
\newcommand{\Qb}{\mathbb{Q}}
\newcommand{\Zb}{\mathbb{Z}}

\newcommand{\Dc}{\mathcal{D}}

\newcommand{\ZF}{\ifmmode\mathsf{ZF}\else$\mathsf{ZF}$\fi}
\newcommand{\ZFC}{\ifmmode\mathsf{ZFC}\else$\mathsf{ZFC}$\fi}

\newcommand{\e}{\varepsilon}

\def\james{first author}%

\theoremstyle{theorem}
\newtheorem{theorem}{Theorem}[section]

\newtheorem{thm}[theorem]{Theorem}
\newtheorem{prop}[theorem]{Proposition}
\newtheorem{lem}[theorem]{Lemma}

\newtheorem{cor}[theorem]{Corollary}
\newtheorem{fact}[theorem]{Fact}

\theoremstyle{definition}
\newtheorem{definition}[theorem]{Definition}
\newtheorem{pseudodefinition}[theorem]{Pseudodefinition}

\theoremstyle{remark}

\usepackage{fullpage}

\title{A formula for any real number, maybe}%
\author{James E. Hanson and Connor Watson}

\begin{document}

\maketitle

\begin{abstract}
  We discuss how to write down three specific natural numbers $A$, $B$, $C$ such that for any real number $r$ you've probably ever thought of, it is consistent with \ZFC{} set theory that $r = \log(\sup_{x_0,x_1 \in \Rb} \inf_{x_2 \in \Rb} \sup_{x_3 \in \Rb}\allowbreak\inf_{x_4 \in \Rb}\sup_{m \in \Nb}\inf_{n_0,\dots,n_{A} \in \Nb} x^2_0 [(n_0 - 2)^2  +  \allowbreak(n_1-m)^2 + n_2  + (n_B - n_C)^2 + n_3 \sum_{k=0}^4 ( x_k - \frac{n_{k+5}}{1+n_4} + n_4)^2  + \sum_{i,j = 0}^B (n_{9+2^i3^j} - n_i^{n_j})^2])$. We also discuss why it's possible, assuming the existence of certain large cardinals, for there to be a real number $s$ which cannot be the value of this formula for our particular $A$, $B$, $C$. This involves set-theoretic mice.

\end{abstract}

\section*{Introduction}

Familiar mathematical constants are often defined or computed in terms of some infinitary process (such as a limit or integral) applied to a firmly concrete, finitary formula involving only rational numbers. Many of these expressions can be rewritten in terms of infima and suprema, such as in the following:
\begin{align*}
e &= \sup_{n \in \Nb} \left( 1 + \frac{1}{n+1} \right)^{n+1} \\
\pi &= 4\inf_{n \in \Nb} \sum_{k=1}^{2n+1} \frac{(-1)^k}{2k+1} \\
\gamma &= \sup_{n \in \Nb} \inf_{m \in \Nb}\left(  \sum_{k=0}^n \frac{1}{k+1} - \sum_{\ell=0}^m\frac{n}{m + \ell n + 1} \right)
\end{align*}
Formulas such as these are so banal in the context of mathematics that it's easy to think that they represent absolutely specific quantities. It would be absurd to suggest that a statement like $\sup_{n \in \Nb} ( 1 + \frac{1}{n+1} )^{n+1} \leq 3$ might be unresolvable by the ordinary methods of mathematical reasoning. Nevertheless, there are statements of the same form which do end up being nebulous in this sense. Using G\"odel's incompleteness theorem and the Matiyasevich--Robinson--Davis--Putnam theorem and some related results, for any given reasonably strong system of axioms, it's possible to show that there are natural numbers $D$, $E$, and $F$ such that the value of
\begin{equation}
  \inf_{n_0,\dots,n_D \in \Nb}
 \left[ %
      (n_0-2)^2  %
      + (n_E - n_F)^2
      + \sum_{i,j = 0}^E (n_{2^i3^j} - n_i^{n_j})^2 %
\right]
   \tag{$\star$}
\end{equation}
 is independent of the given axioms.\footnote{This can be proven with an argument very similar to the proof of \cref{lem:indic}.}

This paper discusses a similar formula whose value is in some sense `even more independent' than the value of $(\star)$. Specifically, we are going to show that for some particular choice of natural numbers $A$, $B$, and $C$, the formula
\begin{equation}
  \sup_{x_0,x_1 \in \Rb} \inf_{x_2 \in \Rb} \sup_{x_3 \in \Rb}\inf_{x_4 \in \Rb}\sup_{m \in \Nb}\inf_{n_0,\dots,n_{A} \in \Nb} x^2_0 \left[
    \begin{matrix*}[l]
      \phantom{+}(n_0 - 2)^2  +  (n_1-m)^2\\
      + n_2 + (n_B - n_C)^2\\
      + n_3 \sum_{k=0}^4 ( x_k - \frac{n_{k+5}}{1+n_4} +n_4)^2 \\
      + \sum_{i,j = 0}^B (n_{9+2^i3^j} - n_i^{n_j})^2
    \end{matrix*}
    \right] \tag{$\dagger$}
  \end{equation}
can consistently be any reasonably simple non-negative real number. Although, as we'll see, the relevant notion of `reasonable simplicity' is actually extremely broad, encompassing, for instance, every real number described by some entry on the \emph{On-Line Encyclopedia of Integer Sequences} and all but two or three real numbers with their own \emph{Wikipeidia} articles. It includes all computable real numbers and many non-computable real numbers such as Chaitin's constant $\Omega$, the halting oracle $0'$, and $\sum_{n =1}^\infty \frac{1}{\Sigma(n)}$, where $\Sigma(n)$ is the busy beaver function. It covers, broadly speaking, every real number that one can prove the existence of in \ZFC{} set theory.

But what do we mean when we say that the behavior of $(\dagger)$ is `more independent' than the behavior of $(\star)$? There are actually two senses in which this is true. The easier sense is just that the set of values $(\dagger)$ can take on is larger. Depending on how a formula like $(\star)$ is constructed, it may only consistently take on a small finite number of values, possibly only two.

The subtler sense in which the value of $(\dagger)$ is more independent than the value of $(\star)$ has to do with the nature of independence results. One fairly robust way of conceptualizing independence results is the `multiverse perspective,' originally pioneered by Hamkins in \cite{HAMKINS2012}. An axiomatic system like Peano arithmetic or Zermelo--Frankel set theory are analogous to the group axioms or the field axioms in that they specify a \emph{kind} of object, rather than a single object. Different models of \ZFC{} can be thought of as different `mathematical universes' between which certain mathematical facts may differ. And, in particular, in order to see different values of the formula $(\star)$, one must go to mathematical universes that are drastically different. In order for the value of $(\star)$ in universe $U_0$ to be different from the value of $(\star)$ in universe $U_1$, it needs to be the case that $U_0$ and $U_1$ disagree about what the natural numbers are.\footnote{The existence of such $U_0$ and $U_1$ follows from the combination of G\"odel's incompleteness theorem and G\"odel's completeness theorem.} For instance it might be the case that $\Nb$ in $U_0$ is a strict initial segment of $\Nb$ in $U_1$. By contrast, the value of $(\dagger)$ can be changed by merely passing to a universe with a different set of real numbers. This is important for formalizing the central claim of this paper because if $U_0$ and $U_1$ don't agree about what the natural numbers are, then real numbers in $U_0$ and real numbers in $U_1$ aren't even the same kind of object. $U_0$ and $U_1$ both have things that they think of as the real number $\pi$, since you can prove that such a number exists in \ZFC{}, but $\pi^{U_1}$ might have digits indexed by natural numbers that don't even exist in $U_0$.

Set theory tends to focus on what are called \emph{well-founded} models of \ZFC{}, which are rich enough to capture a wide variety of phenomena, but tractable enough to study in depth. Well-founded models of \ZFC{} are somewhat analogous to Noetherian rings, both in the role they play in their respective fields, but also on the technical level that they are defined by a finiteness condition. A model of \ZFC{} is well-founded if it does not have an (external) infinite descending sequence of set membership:
\[
x_0 \ni x_1 \ni x_2 \ni x_3 \ni x_4 \ni x_4 \ni \dots
\]
 Well-founded models agree about what the natural numbers are, but may have different sets of real numbers. This is second sense in which $(\dagger)$ is `more independent' than $(\star)$; the value of $(\dagger)$ varies more than the value of $(\star)$ and it does so over a the narrower, well-behaved class of well-founded models, which leave the value of $(\star)$ fixed.

\subsection{How could you possibly prove that?}%

The proof of our main theorem is very similar to previous work of the \james{} author in \cite{AnyFunction}, in which it was shown that a similar expression (albeit one involving only polynomials with integer coefficients) yields a function whose Lebesgue measurability is independent of \ZFC{}. This paper is intended to be vaguely self-contained, but it would be a disservice to not highlight the relationship between it and \cite{AnyFunction}.  The general shape of both arguments involves three components: 
\begin{enumerate}
  \item\label{set-thy} A set-theoretic argument that there is some `computably definable' subset of Cantor space, $2^\omega$, with some desired sensitivity to set-theoretic assumptions (where $\omega$ is set theory notation for the set of natural numbers).
\item\label{des-set-thy} A standard argument using effective descriptive set theory that this can be translated to a `computably definable' subset of $\Rb$ or $\Rb^k$. 
\item\label{num-thy} A (lightly number-theoretic) argument that this definition can be implemented by some (relatively) simple mathematical expression using ideas similar to those in the negative resolution of Hilbert's tenth problem.
\end{enumerate}

The notion of `computable definability' here is in the sense of effective descriptive set theory: Say that an open subset $U$ of $\Rb^k$ is \emph{computably open} if there is a computer program that outputs a list of rational vectors $\vec{x}_n \in \Rb^k$ and rational radii $r_n > 0$ such that $U = \bigcup_{n \in \Nb} B_{r_n}(\vec{x}_n)$ (where $B_{r}(\vec{x})$ is the open ball of radius $r$ centered at $\vec{x}$). Say that a sequence $(U_m)_{m \in \Nb}$ of open sets is \emph{uniformly computably open} (or a \emph{uniformly computable} sequence of open sets) if there is a computer program that takes a natural number $m$ as input and then prints lists $\vec{x}_{n,m}$ and $r_{n,m}$ as before such that $U_m = \bigcup_{n \in \Nb}B_{r_{n,m}}(\vec{x}_{n,m})$. Finally, say that a $G_\delta$ subset $X \subseteq \Rb^k$ is \emph{computably $G_\delta$} if there is a uniformly computably open sequence $(U_m)_{m \in \Nb}$ such that $X = \bigcap_{m \in \mathbb{N}} U_m$.

From here we can define a hierarchy of sets built using projections and complements (where $\pi_{\Rb^k}$ is the projection to $\Rb^k$):
\begin{itemize}
\item A \emph{$\Sigma^1_1$ set} is a set of the form $\pi_{\Rb^k}(X)$, where $X \subseteq \Rb^{k+1}$ is computably $G_\delta$.
\item A \emph{$\Pi^1_n$ set} is the complement of a $\Sigma^1_n$ set.
\item For $n > 1$, a \emph{$\Sigma^1_n$} set is a set of the form $\pi_{\Rb^k}(Y)$, where $Y \subseteq \Rb^{k+1}$ is $\Pi^1_n$.
\item A \emph{projective\footnote{Strictly speaking what we are describing here are the `lightface' projective sets. Normally the term `projective' refers to a larger class of sets called the `boldface' projective sets, but for the sake of simplicity, we won't discuss these.} set} is a set that is $\Sigma^1_n$ or $\Pi^1_n$ for some $n$.
\end{itemize}
Now, these definitions may seem weirdly specific, but these classes of sets are actually fairly robust (for instance, each is closed under finite unions and intersections as well as `computable' countable unions and intersections) and moreover we claim (without proof) that the vast majority of subsets of $\Rb^k$ that are conventionally thought of as `explicitly defined' (specifically without using the axiom of choice) are in this hierarchy (and indeed are almost always $\Sigma^1_1$).

In the rough outline given above, we started with Cantor space ($2^\omega$) in step~\ref{set-thy} but in the previous paragraph we only discussed $\Rb$. This is because, from the point of view of set theory and (effective) descriptive set theory, $2^\omega$ and $\Rb$ are essentially isomorphic.\footnote{So much so that set theorists will regularly refer to elements of $2^\omega$ as `reals' or `real numbers.'} Specifically, one can reasonably define an analogous hierarchy of sets (e.g., computably open/$G_\delta$, $\Sigma^1_n$, $\Pi^1_n$) for $2^\omega$ and then show that there is a bijection $f : \Rb \to 2^\omega$ with the property that any set $X \subseteq \Rb$ is $\Sigma^1_n$ if and only if the image $f(X)$ is $\Sigma^1_n$ (and so clearly the same holds for $\Pi^1_n$).\footnote{$f$ does not preserve computable openness, because this would imply that $\Rb$ and $2^\omega$ are homeomorphic, which they are not.} The existence of this isomorphism is essentially the content of step~\ref{des-set-thy}.

The key set-theoretic fact that allows for the main result of \cite{AnyFunction} is that it is consistent with \ZFC{} that there is a well-ordering $<_L$ of $2^\omega$ such that $\{(x,y) \in (2^\omega)^2 : x <_L y\}$ is both a $\Sigma^1_2$ set and a $\Pi^1_2$ set. This was a seminal result of G\"odel in the 1930s, which he showed by defining the universe of `constructible' sets, $L$, which in a certain sense is a minimal model of \ZFC{}. $L$ and related models of \ZFC{} also play a pivotal role in the main result of this paper, and we will discuss them in more depth in \cref{sec:encoding}.

As mentioned before, different models of \ZFC{} will in general have different sets of real numbers. (When building $L$, for instance, it's entirely possible that a real number might just not show up in $L_\alpha$ for any $\alpha$.) Suppose that we have a computer program $P$ that we are using to define a $\Sigma^1_n$ subset of $\Rb$ and suppose we have two universes of sets (i.e., models of \ZFC{}), $U_0$ and $U_1$. If we use $P$ to define a set $X_0$ of reals in $U_0$, there is no guarantee that this set will be especially similar to the set $X_1$ of reals in $U_1$ also defined by $P$.

To illustrate this idea, let's pretend for a moment that we can build a model $U_0$ of \ZFC{} in which\footnote{To be clear, this is not at all possible.} $\Rb = \Qb$ and an ordinary model $U_1$ in which $\Qb$ is a proper subset of $\Rb$. Given a reasonable enumeration $(q_i)_{i \in \Nb}$ of $\Qb$, it is fairly trivial to write down a computer program $P$ witnessing that $((-\infty,q_i)\cup (q_i,\infty))_{i \in \Nb}$ is a uniformly computably open sequence. If we now use $P$ as a specification of a computably $G_\delta$ set, it will evaluate to the empty set in $U_0$ and to the very much non-empty set of irrational numbers in $U_1$.

This may seem like a silly example, but it's not that far off from something that does actually happen with models of \ZFC{}. There is a fixed computer program $P$ that defines the set $\Rb \cap L$ (i.e., the set of real numbers that exist inside $L$) as a $\Sigma^1_2$ set in any model of \ZFC{}. Considering the complement of this, we now have a recipe for a $\Pi^1_2$ set which is empty in $L$ but non-empty in models $U$ of \ZFC{} in which $\Rb$ is a proper superset of $\Rb \cap L$.

Another example (which is related to concepts that will show up later in \cref{sec:any-real}), has to do with certain specific real numbers that cannot exist inside $L$. There are two ways for a real number to avoid being an element of $L$: It can be `too random' or it can `know too much.' Forcing can add real numbers that are `random' or `generic' (although the specific details depend a lot on the kind of forcing in question), but this means that the produced real numbers added by this are not specific. 

The real numbers that `know too much,' by contrast, can be very specific. The simplest example of this is $0^\sharp$ (read `zero sharp'), which is a real number that in some sense codes a `blueprint' for building $L$. Specifically, this real number allows one to compute a truth predicate for $L$, which means that if it were an element of $L$, $L$ would be able to define its own truth predicate, contradicting Tarski's undefinability theorem. Since $0^\sharp$ is very specific, it can be described in a projective way. In particular, there is a $\Pi^1_2$ definition which yields the set $\{0^\sharp\}$ if $0^\sharp$ exists and the set $\varnothing$ otherwise.

\section{\texorpdfstring{Encoding a real number into $\Rb$ itself}{Encoding a real number into ℝ itself}}
\label[sec]{sec:encoding}

For the purpose of the main result of this paper, we need to show a certain set-theoretic fact, namely that it is possible to `definably encode' a given real number $r \in \Rb \cap L$ into the set of reals $\Rb \cap M$ in a specially constructed well-founded model $M$ of \ZFC{}.

The idea of defining a real number with the property that the real number which it defines is independent of \ZFC{} is not a new one, and it is not particularly difficult if you are willing to relax your definition of `defining.' For example, consider the following definition of a real number $r\in 2^{\omega}$:
$$r(n)=1\iff 2^{\aleph_n}=\aleph_{n+1}.$$

By Easton's theorem, for any set $X\subseteq\Nb,$ there is a model of \ZFC{} in which the GCH holds at exactly those $\aleph_n$'s for which $n\in X.$ Hence, not only is the real number that it defines independent of \ZFC{}, but it can consistently be any real number. We say that we have ``coded $r$ in the continuum function." %

This, of course, feels a bit cheap. One reason it might feel that way is that this method appeals to the nature of cardinal arithmetic, which certainly does not define a projective singleton. The main idea behind the proof of our main theorem is to code the bits of the real number $r$ just as we did with continnum function, but to code it into a different structure which can be done in a projective way. We use the structure of $L$ to do this.

The relationship between $L$ and other well-founded models of \ZFC{} is analogous to the relationship between $\Qb$ and other fields of characteristic $0$ in that $L$ has a unique embedding into any other well-founded model of \ZFC{} (with the same ordinal numbers). $L$ is defined by a transfinite inductive construction inspired in some way by the \emph{cumulative hierarchy}:
\begin{itemize}
\item $V_0 = \varnothing$.
\item $V_{\alpha + 1} = \mathcal{P}(V_\alpha)$, where $\mathcal{P}(X)$ is the power set of $X$.
\item $V_{\lambda} = \bigcup_{\alpha < \lambda}V_\alpha$ for $\lambda$ a limit ordinal.
\item $V=\bigcup_{\alpha\in\mathrm{On}} V_\alpha$, where $\mathrm{On}$ is the class of ordinal numbers.
\end{itemize}
This is the maximal construction of this form, in that it grabs the full power set at each stage. G\"odel's original idea with defining $L$ was to define a `minimal' construction of a model of \ZFC{}, adding sets as conservatively as possible:
\begin{itemize}
\item $L_0 = \varnothing$.
\item $L_{\alpha + 1}$ is the set of all first-order definable subsets of $L_\alpha$.
\item $L_{\lambda} = \bigcup_{\alpha < \lambda} L_\alpha$ for $\lambda$ a limit ordinal.
\item $L=\bigcup_{\alpha\in\mathrm{On}} L_\alpha.$
\end{itemize}
It's clear from this definition that a model of \ZFC{} needs to be closed under these operations, since it satisfies the axiom scheme of separation, but the surprising thing is that these operations are also enough to build a model of \ZFC{}. In fact the resulting model is \emph{absolute} in the sense that (in the context of well-founded models) $L$ is uniquely determined by the length of its class of ordinals (which can vary, as we'll discuss in \cref{sec:what-is-large}). $L$ is highly structured, and a lot more can be said about it than about arbitrary models of \ZFC{}. This is a big part of what makes it possible to prove the main result of this paper (and similarly the main result of \cite{AnyFunction}).

A major topic in modern set theory is studying generalizations of this construction to build other \emph{inner models}. Perhaps the easiest generalization is $L[X]$.

\begin{definition}\label[definition]{defn:L-X}
  Given a set $X$, $L_\alpha[X]$ and $L[X]$ are defined in the same way as $L_\alpha$ and $L$, respectively, except for the following:
  \begin{itemize}
  \item $L_{\alpha + 1}[X]$ is the set of all subsets of $L_\alpha[X]$ that are first-order definable using $X$ as a parameter.
  \end{itemize}
\end{definition}

Although $X$ can be any set in this definition, it is often a real number. Returning to the metaphor with fields, just as how the field extension $\Qb(\sqrt{2})$ is the unique smallest field containing $\sqrt{2}$, when $x$ is a real number or an element of $2^\omega$, $L[x]$ is the smallest model of \ZFC{} containing the ordinals and $x$. %

An important aspect of the construction of $L[x]$ is that it is compatible with the notion of definability discussed earlier when restricted to the real numbers.

\begin{fact}\label[fact]{fact:x-L-y}
  The relation $x\in L[y]$ is a $\Sigma^1_2$ subset of $\mathbb{R}^2$ \cite[Thm.~13.9]{Kanamori2003}.
\end{fact}

The relation $x \in L[y]$ is interesting in its own right in that it means that $y$ `contains all of the information contained in $x$' in some precise sense. Let us define a preorder on $\Rb$ by saying that for real numbers $x$ and $y$, $$x\leq_c y\iff x\in L[y].$$ Say that $x\equiv_c y$ if and only if $x\leq_c y$ and $y\leq_c x$. This can be shown to define an equivalence relation on $\mathbb{R}$, and so we may form the quotient of $\mathbb{R}$ by this equivalence relation. These equivalence classes are called \textit{constructibility degrees}, or \textit{$c$-degrees}. The partial order of $c$-degrees in a model of \ZFC{} is analogous (although not perfectly so) to the lattice of subfields of a field $K$. The $c$-degrees can also be thought of as a set-theoretic version of oracle computability, and in fact the notation $0^\sharp$ was meant to be similar to the computability-theoretic notation for the halting oracle, $0'$. And just like the halting oracle, there is a relativized version, $x^\sharp$, which has the same relationship to $L[x]$ that $0^\sharp$ has to $L$. In particular, $x^\sharp$ cannot be an element of $L[x]$ and so the $c$-degree of $x^\sharp$ is strictly larger than the $c$-degree of $x$.

It is the structure of the $c$-degrees which we exploit to define the real number in question. As it turns out, just as the GCH pattern can be forced to be almost whatever we would like in different models of \ZFC{}, the exact structure of the $c$-degrees can be quite malleable, too. This is made precise by the following theorem of Adamowicz, which essentially states that the lattice of $c$-degrees can consistently be any well-founded upper semilattice with a top element, provided that the lattice is `reasonably simple' (i.e., lives in $L$):

\begin{fact}[Adamowicz \cite{Adamowicz1977}]\label[fact]{fact:Adamowicz}
    Let $\mathcal{D}$ be a well-founded upper semilattice in $L$ with a top element $\mymathbb{1}.$ Then there is a forcing extension $L[G]$ of $L$ in which the lattice of $c$-degrees of $L[G]$ are isomorphic in $L[G]$ to $\mathcal{D}$ enlarged by a bottom element $\mymathbb{0}.$
\end{fact}

It is this malleability which allows us to code the bits of our real number in a way which is model-dependent in a lightface projective (specifically $\Sigma^1_4$) way. We may now prove our main set-theoretic result.

\begin{figure}
  \centering
   \hspace{-2.5em} \includegraphics[width=0.9\textwidth]{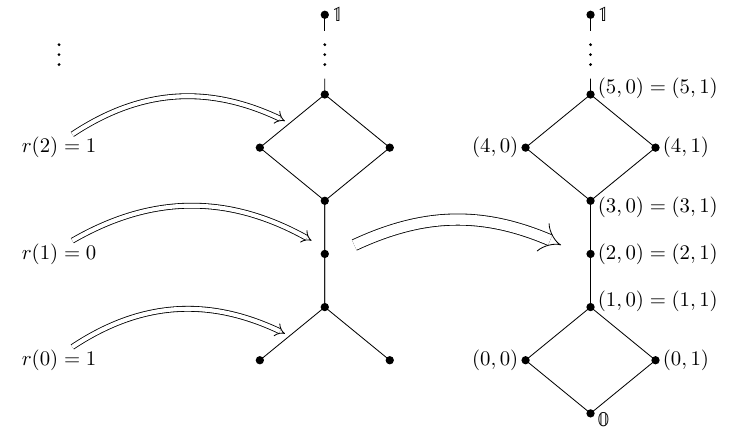}
  %
  \caption{ The $n$th bit of $r$ determines whether the corresponding level of the lattice $\Dc$ has width $1$ or $2$.
  }
  \label{fig:degree}
\end{figure}

\begin{thm}\label[thm]{thm:main-set-theory}
  There is a fixed $\Sigma^1_4$ definition $\Phi$ such that for any $r \in 2    ^\omega \cap L$, there is a forcing extension $L[G]$ in which $\Phi^{L[G]} = \{r\}$.
\end{thm}
\begin{proof}
    Fix $r\in 2^\omega\cap L$. Let $\Dc_0$ be the partial order $\omega \times 2_{\text{antichain}}$ with the lexicographic ordering (where $2_{\text{antichain}}$ is an antichain of size $2$). Let $E$ be the equivalence relation on $\Dc_0$ generated by the set $\{((2n,0),(2n,1)) : r(n) = 0\} \cup \{((2n+1,0),(2n+1,1)) : n < \omega\}$. Finally, let $\Dc = \Dc_0/E + \{\mymathbb{1}\}$, where $\mymathbb{1}$ is added as a new top element. (This is the lattice in the middle of Figure~\ref{fig:degree}.) 

$\mathcal{D}$ is clearly a well-founded upper semilattice with a top element, so by \cref{fact:Adamowicz}, there is a forcing extension $L[G]$ of $L$ in which the lattice of $c$-degrees is isomorphic to $\{\mymathbb{0}\}+\mathcal{D}$ (i.e., $\Dc$ enlarged by a bottom element $\mymathbb{0}$). Let $\Phi$ be the formula which (informally) states that, in $L[G],$``$r$ is the bottom $c$-degree, the bits of $r$ are encoded into the $c$-degrees in the way we have just done, and there are no other $c$-degrees." The construction ensures that $\Phi^{L[G]}=\{r\}.$

We now state what the formula $\Phi$ is more carefully, and show that it is $\Sigma^1_4.$ First, we encode the lattice as a family $(x_{i})_{i \in \{\mymathbb{0}\} \sqcup \omega \times 2 \sqcup \{\mymathbb{1}\}}$ of real numbers, as in the lattice on the right of {Figure~\ref{fig:degree}}. Let $\varphi(x,(x_i))$ be the formula that says for each $n < \omega$: %
\begin{enumerate}
\item $x(n) = 1 \Leftrightarrow x_{(2n,0)} \equiv_c x_{(2n,1)}$
\item  $x(n) = 0 \Leftrightarrow x_{(2n,0)} \perp_c x_{(2n,1)}$ (where $y \perp_c z$ means that $y \not \leq_c z$ and $z \not \leq_c y$)
\item $x_{(2n+1,0)} \equiv_c x_{(2n+1,1)}$
\end{enumerate}
 By \cref{fact:x-L-y}, this is a computably countable Boolean combination of $\Sigma^1_2$ facts, so it is $\Pi^1_3$.

Let $\psi((x_i))$ be the conjunction of the following conditions:
\begin{enumerate}
\item For all $z$, there exists an $i \in \{\mymathbb{0}\} \sqcup \omega \times 2 \sqcup \{\mymathbb{1\}}$ such that $z \equiv_c x_i$.
\item For all $n < m < \omega$ and $i,j < 2$, $x_{\mymathbb{0}} <_c x_{(n,i)} <_c x_{(m,j)} <_c x_{\mymathbb{1}}$.
\end{enumerate}
In 1, the condition after `for all $z$' is a Boolean combination of $\Sigma^1_2$ conditions, and so is itself $\Pi^1_3$. The condition in 2 is a Boolean combination of $\Sigma^1_2$ conditions, and so is likewise $\Pi^1_3$. This means that $\psi((x_i))$ is a $\Pi^1_3$ condition.

Our final formula defining the real $x$ is the formula $$\Phi(x)=\exists (x_{(n,i)})\;[\varphi(x,(x_{(n,i)}))\land\psi((x_{(n,i)}))]$$

This formula is $\Sigma^1_4$ since there is an existential quantifier in front of a $\Pi^1_3$ formula, and since $\{y:\Phi(y)\}=\{x\}$, we have that $\{x\}$ is a $\Sigma^1_4$ set.
\end{proof}

\section{A small Diophantine computer}%

If you are an even slightly observant reader, you surely noticed that the main bodies of the formulas $(\ast)$ and $(\dagger)$ in the introduction bear a strong resemblance to each other. This is no coincidence. The common elements of these two formulas essentially specify a small computer made from an (exponential) Diophantine expression. The changes in indexing (e.g., $ 2^i 3^j$ vs.\ $9 + 2^i3^j$) just represent slightly different choices of `memory allocation,' and the additional elements in $(\dagger)$ are there to accommodate an input value (i.e., $m$) and some interaction with the real-valued variables $x_0,\dots,x_3$. The triples of natural numbers $(A,B,C)$ and $(D,E,F)$ represent different choices of computer programs. %

Something that is important to note about the arguments in this section is that if they feel a bit tedious and arbitrary, that's because they are. Much like in the design of an actual physical computer, there is an  incredible amount of leeway in the construction of a Diophantine computer. For instance, using arguments similar to those in \cite{AnyFunction}, it is possible to replace the expression in the square brackets of our main formula with a polynomial with integer coefficients. Our construction here is mostly optimized for the sake of producing a short mathematical expression.\footnote{Although certainly not as short as possible. For instance, with a longer version of the proof of \cref{lem:ce-set-indic}, one can replace $(n_0-2)^2 + (n_1-m)^2$ with $(n_0 - 2 -m)^2$. Likewise, it is possible to get away with $(x_k^2-\frac{n_{k+5}}{1+n_4})^2$, rather than $(x_k-\frac{n_{k+5}}{1+n_4} + n_4)^2$.} As you will see in the proof of \cref{cor:ce-set-indic-2}, a large fraction of the $n_k$'s are entirely unnecessary. Similarly, the variable $n_0$ could just be replaced with the constant $2$, but the resulting formula, while conceptually simpler, is longer to actually write down.

The following \cref{lem:ce-set-indic} is heavily based on ideas from early work on Hilbert's tenth problem, specifically Martin Davis, Hilary Putnam, and Julia Robinson's reduction of Diophantine equations to purely exponential equations in \cite{Davis1961}, and Robinson's enumeration of Diophantine equations in \cite{Robinson1971Hilbert10}. The  crucial idea in \cite{Davis1961} is this:
\begin{itemize}
\item Multiplication, $x\cdot y$, can be computed using exponentiation by $(2^x)^y = 2^{x\cdot y}$.
\item  Likewise, addition, $x+y$, can be computed using exponentiation and multiplication by $2^x \cdot 2^y = 2^{x+y}$, which can in turn be computed using only exponentiation by $(2^{2^x})^{2^y} = 2^{2^x 2^y} = 2^{2^{x+y}}$.
\item By putting these together, we can write arbitrary Diophantine equations using only exponentiation and the constant $2$. For example, $x+y-2xy = 0$ if and only if $(2^{2^x})^{2^y} = 2^{2^{x+y}} = 2^{2^{2xy}} = 2^{2^{xy}2^{xy}} = (2^{(2^x)^y})^{(2^x)^y}$.
\end{itemize}

In preparation for the proof, let's call a sequence $\vec{n} = (n_k)_{k \leq A}$ of natural numbers \emph{well-formed} if $n_0 = 2$ and for each $i,j \in \Nb$, if $9+2^i3^j \leq A$, then $n_{9+2^i3^j} = n_i^{n_j}$. Note that, since $(i,j) \mapsto 9 + 2^i3^j$ is an injection, any well-formed sequence $(n_k)_{k \leq A}$ can be extended (typically non-uniquely) to a well-formed sequence $(n_k)_{k \leq A'}$ for any $A' \geq A$. Also note that any non-empty initial segment of a well-formed sequence is well-formed. See Figure~\ref{fig:memory} for an explanation of the rationale behind this definition.

Let $Y$ be the set of variables $y_k$ for $k \notin \{0\}\cup\{9+2^i3^j : i,j \in \Nb\}$. Given a polynomial $p$ with variables in $Y$ and given a sequence $(n_k)_{k \leq A}$ (satisfying that for every variable $y_k$ in $p$, $k\leq A$) we will write $p(\vec{n})$ for the value of $p$ under the variable assignment $y_k \mapsto n_k$. (Note that not every $n_k$ in the given sequence will be assigned to a variable in $p$.)

\begin{lem}
  \label[lem]{lem:ce-set-indic}
  $ $
  \begin{enumerate}
  \item For any polynomial $p \in \Nb[Y]$ (i.e., with coefficients in $\Nb$), there is a natural number $D$ such that if $\vec{n} = (n_k)_{k \leq D}$ is well-formed, then \(n_D = 2^{2^{p(\vec{n})}} \).
  \item For any polynomial $p \in \Zb[Y]$, there is a pair of natural numbers $B,C$ with $B \geq C$ such that for any well-formed sequence $(n_k)_{n \leq B}$, $n_B = n_C$ if and only if $p(\vec{n}) = 0$. %
  \end{enumerate}
\end{lem}
\begin{proof}
For 1, first note that if $p = 0$, then we can just take $D = 0$ (since $2^{2^0} = 2$).
  
  Next, we claim that if $p$ is a monomial with coefficient $1$, then we can find a natural number $E_p$ such that if $\vec{n}$ is well-formed, then $n_{E_p} = 2^{p(\vec{n})}$. This follows by induction on the degree of $p$. For $p = 1$, we can take $E_p = 2$. If $p = y_kq$ and $E_q$ is already known, then we can take $E_p = 9+2^{E_q}3^{k}$. This works because (in a well-formed sequence), we have that $n_{E_p} = (n_{E_q})^{n_k}=(2^{q(\vec{n})})^{n_k} = 2^{n_kq(\vec{n})}$.

  Finally, for general non-zero polynomials, note that any such polynomial can be written as a sum of monomials with coefficient $1$. Again, we proceed by induction. Let $q$ be a monomial with coefficient $1$ and let $r$ be a polynomial satisfying that for any well-formed $\vec{n}$, $n_{D_r} = 2^{2^{r(\vec{n})}}$. Now let $p = r + q$. We may take $D_p$ to be $9 + 2^{D_r}3^{E_q}$. For any well-formed $\vec{n}$, we now have that 
  \[
    n_{D_p} = (n_{D_r})^{n_{E_q}} = (2^{2^{r(\vec{n})}})^{2^{q(\vec{n})}} = 2^{2^{r(\vec{n})}2^{q(\vec{n})}} = 2^{2^{r(\vec{n}) + q(\vec{n})}} = 2^{2^{p(\vec{n})}}.
  \]
  So we have by induction that for any polynomial $p \in \Nb[Y]$, we can find the required $D_p$.

  For 2, first note that by an elementary argument we have that for any $p \in \Zb[Y]$, there are $q_0,q_1 \in \Nb[Y]$ such that $p = q_0 - q_1$. By part 1, we can find $B,C \in \Nb$ such that for any well-formed $\vec{n}$, $n_B = 2^{2^{q_0(\vec{n})}}$ and $n_C = 2^{2^{q_1(\vec{n})}}$. Without loss of generality, we may assume that $B \geq C$. Now it is immediate that for any sufficiently long well-formed $\vec{n}$,
  \[
    p(\vec{n}) = 0 \Leftrightarrow q_0(\vec{n}) = q_1(\vec{n}) \Leftrightarrow 2^{2^{q_0(\vec{n})}} = 2^{2^{q_1(\vec{n})}} \Leftrightarrow n_B = n_C. \qedhere
  \]
\end{proof}

\begin{figure}
  \centering
  \includegraphics[width=\textwidth]{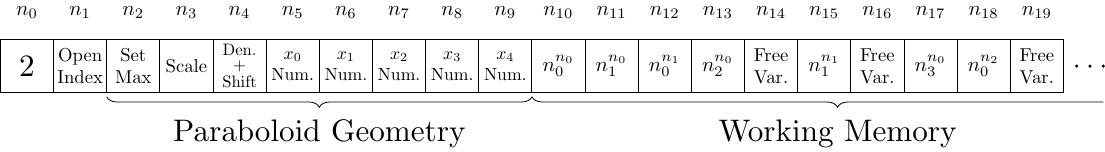}
  \caption{`Memory layout' of our Diophantine formula. $n_0$ is a fixed reference value of $2$. $n_1$ is the index of the open set being constructed. $n_2$ helps ensure that the maximum value of the function is $1$. $n_3$ is the scale of the paraboloid. $n_4$ sets the denominators of the coordinates of the vertex of the paraboloid and shifts them in the negative direction. $n_5$ through $n_9$ set the numerators of the coordinates of the vertex. Addresses of the form $n_{9+2^i3^j}$ store the value of $n_i^{n_j}$ and the remaining addresses are free variables.}
  \label{fig:memory}
\end{figure}

Now using \cref{lem:ce-set-indic} and the fact that computably enumerable sets are the same as Diophantine sets, we get the following corollary. 

\begin{cor}
  \label[cor]{cor:ce-set-indic-2}
  For any computably enumerable set $W \subseteq \Nb^9$, there are natural numbers $B,C$ with $B \geq C$ such that  $(n_1,n_2,\dots,n_{9})$ is an element of $W$ if and only if it extends to a well-formed sequence $(n_k)_{k \leq B}$ satisfying $n_B = n_C$.
\end{cor}
\begin{proof}
  Fix a computably enumerable set $W \subseteq \Nb^9$. Let $J(x,y) = x+(x+y)^2$, and recall that $J$ is an injection from $\Nb^2$ into $\Nb$. Let $q(y_1,y_2,\dots,y_{9}) = J(y_1,J(y_2,\allowbreak J(y_3,J(y_4,J(y_5,J(y_6,J(y_7,J(y_8,y_9))))))))$, and note that this is an injection from $\Nb^9$ into $\Nb$. Let $W' = \{q(y_1,y_2,\dots,y_{9}) : (y_1,y_2,\dots,y_{9}) \in W\}$. It is immediate that $W'$ is a computably enumerable subset of $\Nb$. By the proof of the negative resolution of Hilbert's tenth problem \cite{Matiyasevich_1970}, we can find a polynomial $p(w,z_1,\dots,z_m) \in \Zb[w,z_1,\dots,z_m]$ such that for any $w \in \Nb$, $w \in W'$ if and only if there exists $(z_1,\dots,z_m) \in \Rb^m$ such that $r(w,z_1,\dots,z_m) = 0$. Now consider the polynomial
  \[
    p = r(q(y_1,y_2,\dots,y_{9}),y_{9+5\cdot 1},y_{9+5\cdot 2},\dots,y_{9+5\cdot m}),
  \]
  where $y_{9+5\cdot 1},\dots,y_{9+5\cdot m}$ were merely chosen to be distinct variables in $Y$ (i.e., with indices not of the form $9+2^i3^j$). By construction, we have that $(y_1,y_2,\dots,y_{9}) \in W$ if and only if there are $(y_{9+5\cdot 1},\dots,y_{9+5\cdot m}) \in \Nb^m$ such that $p = 0$.

  By \cref{lem:ce-set-indic}, we can find $B,C \in \Nb$ with $B \geq C$ such that for any well-formed sequence $(n_k)_{k \leq B}$, $p(\vec{n}) = 0$ if and only if $n_B = n_C$. Since every tuple $(n_1,n_2,\dots,n_{9},n_{9+5\cdot 1},\dots,n_{9+5\cdot m})$ extends to a well-formed sequence $(n_k)_{k \leq B}$, we have that $(n_1,n_2,\dots,n_{9}) \in W$ if and only if it extends to a well-formed sequence $(n_k)_{k \leq B}$ satisfying $n_B = n_C$, as required.
\end{proof}

\section{Main theorem}
To avoid writing such a large expression over and over again, 
let 
  \[
f_{A,B,C}(x_0,\dots,x_4,m,n_0,\dots,n_A) =   \left[
    \begin{matrix*}[l]
      \phantom{+}(n_0 - 2)^2  +  (n_1-m)^2  \\
      + n_2 + (n_B - n_C)^2 \\
      + n_3 \sum_{k=0}^4 ( x_k - \frac{n_{k+5}}{1+n_4} +n_4)^2 \\
      + \sum_{i,j = 0}^B (n_{9+2^i3^j} - n_i^{n_j})^2
    \end{matrix*}
    \right].
  \]
  We will abbreviate this as $f_{A,B,C}$.

The following lemma is essentially an effective or computable version of the fact that any indicator function of a closed set is an infimum of a family of functions of the form $a(\vec{x}-\vec{b})^2 + c$ with $a,c \geq 0$.
  
\begin{lem}
  \label[lem]{lem:indic}
  For any uniformly computably sequence $(U_m)_{m \in \Nb}$ of non-empty open subsets of $\Rb^5$, there are natural numbers $B$ and $C$ such that
 \[\inf_{n_0,\dots,n_A \in \Nb}f_{A,B,C} =
    1 - \chi_{U_m}(x_0,x_1,x_2,x_3,x_4)\] %
  for each $m \in \Nb$, where $A = 9+2^B3^C$ and $\chi_{U_m}$ is the indicator function of $U_m$. %
\end{lem}
\begin{proof}
  Since the $U_m$ are uniformly computably open and non-empty, we can find a computable list $((\vec{y}_{m,i},r_{m,i}))_{m,i \in \Nb}$ of rational points and radii such that for each $m$, $U_m = \bigcup_{i \in \Nb}B_{r_{m,i}}(\vec{y}_{m,i})$. Let $W$ be the set of all tuples $(n_1,n_2,\dots,n_{9}) \in \Nb^9$ such that either 
  \begin{itemize}
  \item $n_2 = 1$ and $n_3 = 0$ or
  \item $n_2 = 0$ and there exists an $i \in \Nb$ such that for any $(x_0,x_1,\dots,x_4) \in \Rb_{\geq 0}^5$, if
    \(
      n_3\sum_{k=0}^4\left(x_k - \frac{n_{k+5}}{1+n_4} + n_4\right)^2 \leq 1,
    \)
    then $(x_0,x_1,\dots,x_4) \in B_{r_{n_1,i}}(\vec{y}_{n_1,i})$.
  \end{itemize}
  It is not difficult, but somewhat tedious, to show that $W$ is also a computably enumerable set.

  By \cref{cor:ce-set-indic-2}, we can find natural numbers $B,C$ such that  $(n_1,n_2,\dots,n_{9}) \in W$ if and only if any of the following equivalent conditions hold:
  \begin{itemize}
  \item $(n_1,n_2,\dots,n_9)$ extends to a well-formed sequence $(n_k)_{k \leq 9+2^B3^C}$ satisfying $n_B = n_C$.
  \item $(n_1,n_2,\dots,n_9)$ extends to a sequence $(n_k)_{k \leq 9+2^B3^C}$ satisfying that
    \[
      (n_0-2)^2 + \sum_{i,j = 0}^B (n_{9+2^i3^j} - n_i^{n_j})^2 + (n_B - n_C)^2 = 0.
    \]
  \item The same as the previous bullet but with $< 1$ instead of $= 0$.
  \end{itemize}

  \begin{sloppypar}
    It is straightforward to show that for any $m\in \Nb$, $\e > 0$, and $(x_0,x_1,\dots,x_4) \in U_m$, there is $(n_1,n_2,\dots,n_9) \in W$ such that $n_1 = m$, $n_2 = 0$, and $n_3\sum_{k=0}^4(x_k - \frac{n_{k+5}}{1+n_4} + n_4)^2 < \e$.\footnote{This relies on the fact that for any rational $(y_0,\dots,y_4)$, there are $(n_4,n_5,\dots,n_9) \in \Nb^6$ such that $y_k = \frac{n_{k+5}}{1+n_4} -n_4$ for each $k \leq 4$. For a careful version of a similar argument, see the proof of Lemma~4.4 in \cite{AnyFunction}.} This implies that for any $m$ and $(x_0,x_1,\dots,x_4) \in U_m$, we have that $\inf_{n_0,\dots,n_A \in \Nb}f_{A,B,C} = 0$.
  \end{sloppypar}
  Finally we need to show that for any $m \in \Nb$, if $(x_0^2,x_1^2,\dots,x_4^2) \notin U_m$, then $\inf_{n_0,\dots,n_A \in \Nb}f_{A,B,C} = 1$. By our inclusion of tuples satisfying $n_2 = 1$ and $n_3=0$ into $W$, we have that $\inf_{n_0,\dots,n_A \in \Nb}f_{A,B,C} \leq 1$ for any $(m,x_0,\dots,x_4)$.\footnote{This is the only reason for the inclusion of the $n_2$ term.} By the construction of $W$, we have that for any $m \in \Nb$ and $(x_0,\dots,x_4,) \in \Rb^5$, if $(x_0,\dots,x_4) \notin U_{m}$, then $f_{A,B,C} \geq 1$ for every $(n_0,\dots,n_A) \in \Nb^{A+1}$, implying that $\inf_{n_0,\dots,n_A \in \Nb}f_{A,B,C} = 1$, as required.
\end{proof}

\begin{cor}\label[cor]{cor:indic-half}
  For any $\Sigma^1_4$ definition $X$ of a subset of $\Rb$, there are natural numbers $A$, $B$, and $C$ such that
  \[
    g(x_0) =
  \sup_{x_1 \in \Rb} \inf_{x_2 \in \Rb} \sup_{x_3 \in \Rb}\inf_{x_4 \in \Rb}\sup_{m \in \Nb}\inf_{n_0,\dots,n_A \in \Nb}f_{A,B,C}
  \]
  is the indicator function of the set defined by $X$.
\end{cor}
\begin{proof}
  \begin{sloppypar}
    Fix a $\Sigma^1_4$ definition $X$ of a subset of $\Rb$. Let $(U_n)_{n \in \Nb}$ be a sequence of uniformly computably open sets corresponding to $X$. By \cref{lem:ce-set-indic}, $1-\sup_{m \in \Nb}\inf_{n_0,\dots,n_A \in \Nb} f_{A,B,C}$ is the indicator function of a set $Y$ that is the complement of a computably $G_\delta$ subset of $\Rb^5$ with the property that $x_0$ is in the set defined by $X$ if and only if
  \end{sloppypar}
  \begin{itemize}
  \item there exists an $x_1 \in \Rb$ such that for every $x_2 \in \Rb$, there exists an $x_3 \in \Rb$ such that for every $x_4 \in \Rb$, $(x_0,\dots,x_4) \in Y$.
  \end{itemize}
  This is equivalent to saying that $g(x_0)$ is the indicator function of the set defined by $X$.
\end{proof}

Now we finally come to our headline result.

\begin{thm}\label[thm]{thm:main}
  There are natural numbers $A$, $B$, and $C$ such that for any non-negative real $r \in \Rb \cap L$, there is a well-founded model $M$ of \ZFC{} such that in $M$
  \[
    r = 
  \sup_{x_0,x_1 \in \Rb} \inf_{x_2 \in \Rb} \sup_{x_3 \in \Rb}\inf_{x_4 \in \Rb}\sup_{m \in \Nb}\inf_{n_0,\dots,n_{A} \in \Nb} x^2_0 \left[
    \begin{matrix*}[l]
      \phantom{+}(n_0 - 2)^2  +  (n_1-m)^2  \\
      + n_2 + (n_B - n_C)^2 \\
      + n_3 \sum_{k=0}^4 ( x_k^2 - \frac{n_{k+5}}{1+n_4}  + n_4)^2 \\
      + \sum_{i,j = 0}^B (n_{9+2^i3^j} - n_i^{n_j})^2
    \end{matrix*}
    \right].
  \]
\end{thm}
\begin{proof}
      By \cref{thm:main-set-theory}, we have that there is a $\Sigma^1_4$ definition $X$ such that for any non-negative $r \in \Rb \cap L$, there is a forcing extension $L[G]$ such that $X^{L[G]} = \{\sqrt{r}\}$. By \cref{cor:indic-half}, we can find natural numbers $A$, $B$, and $C$ such that $\sup_{x_1 \in \Rb} \inf_{x_2 \in \Rb} \sup_{x_3 \in \Rb}\allowbreak\inf_{x_4 \in \Rb}\sup_{m \in \Nb}\inf_{n_0,\dots,n_{A} \in \Nb} f_{A,B,C}$ is equal to $\chi_{\{\sqrt{r}\}}(x_0)$ when evaluated in $L[G]$. This now implies that, in $L[G]$,
\[
  \sup_{x_0 \in \Rb}x_0^2 \sup_{x_1 \in \Rb} \inf_{x_2 \in \Rb} \sup_{x_3 \in \Rb}\inf_{x_4 \in \Rb}\sup_{m \in \Nb}\inf_{n_0,\dots,n_{A} \in \Nb} f_{A,B,C}
\]
is equal to $(\sqrt{r})^2 = r$.
Since $x_0^2 \geq 0$, this is also equal to
\[
\sup_{x_0,x_1 \in \Rb} \inf_{x_2 \in \Rb} \sup_{x_3 \in \Rb}\inf_{x_4 \in \Rb}\sup_{m \in \Nb}\allowbreak\inf_{n_0,\dots,n_{A} \in \Nb}x_0^2 f_{A,B,C},
\]
which is the formula in the statement of the theorem.
\end{proof}

Using a result of Harrington \citetext{\citealp[Ch.~25]{Miller2017}; \citealp{Harrington1977}}, it is possible to bring this down to a $\Sigma^1_3$ definition, giving a $\sup_\Rb\inf_\Rb\sup_\Rb\sup_\Nb\inf_\Nb$ formula, but the proof of this result is considerably more technical than the already fairly technical proof of \cite[Thm.~1]{Adamowicz1977}. Moreover, we felt that the proof in \cref{sec:encoding} made for a more accessible story. %

\section{What about \emph{any} real number?}
\label{sec:any-real}

Given what we have discussed so far, it is natural to wonder whether the formula in \cref{thm:main} can be made to take on any real numbered value at all. There's a sort of fundamental subtlety with this question, which is that what `any real number' means depends on the set-theoretic universe one is in, but also the result itself relies on passing from one set-theoretic universe to another. The formula can, in some sense, take on any real value, provided that the only real numbers that exist are those in $L$, but in order for it take on those values, it needs to be considered in the context of models of \ZFC{} which have strictly more real numbers than $L$. Given an $r \in \Rb \cap M \setminus L$ for one of these models $M$, is there yet another model of \ZFC{} in which the value of the formula is $r$? Again using \citetext{\citealp[Ch.~25]{Miller2017}; \citealp{Harrington1977}}, it is possible to extend to certain $r$ not in $L$ such that $L[r]$ is `not too much bigger than $L$,'\footnote{Specifically, the relevant condition is that $\omega_1^{L[r]} = \omega_1^L$ (i.e., that $L[r]$ and $L$ agree about what the first uncountable ordinal is).} but a precise characterization of which real numbers can be the unique elements of projectively definable singletons in general is open. At the other end of the spectrum, however, we can show that if $\Rb$ is `sufficiently big,' then the technique used in this paper cannot work for all reals. Specifically, we have the following fact, which we'll take as a black box for the moment.

\begin{fact}\label[fact]{fact:big-enough}
  In the presence of sufficiently large large cardinals,\footnote{It is enough to assume that $\Delta^1_2$-determinacy holds and that $r^\sharp$ exists for every real $r$.} there is a real number $x$ such that for any $y \geq_c x$, $L[x]$ and $L[y]$ have the same first-order theory \citetext{\citealp[Rem.~6.4]{Koellner2009}; \citealp{Kechris1983}}.
\end{fact}


$L[x]$ and $L[y]$ having the same first-order theory is much stronger than this, but the particular consequence of this that is relevant to this paper is that $L[x]$ and $L[y]$ agree about the values of all formulas like the one in our \cref{thm:main}. As such, we get the following:

\begin{prop}\label[prop]{prop:big-doesn't-work}
  If there is a real number $x$ with the property in \cref{fact:big-enough}, then there is a real number $s$ that cannot be the value of the formula in \cref{thm:main} in any model of \ZFC{} of the form $L[y]$.
\end{prop}
\begin{proof}
  Consider $L[x]$ and force to add a new real number $z$. (It doesn't really matter what kind of forcing you use, as long as a new real number shows up.) Let $s = x \oplus z$ (i.e., the number that results from interleaving the digits of $x$ and $z$, which is the least upper bound of $\{x,z\}$ in the pre-order $\leq_c$). Now consider $y \geq_c x$. Assume for the sake of contradiction that the value of the formula in \cref{thm:main} is $s$. This implies that $s \in L[y]$. Since $L[x]$ and $L[y]$ have the same first-order theory, $s$ is the value of the formula in $L[x]$ as well. But this implies that $s$ and therefore $z$ are elements of $L[x]$, which is a contradiction.
\end{proof}

So that's why our method eventually fails. Once $x$ has large enough $c$-degree, the `definable behavior' of $L[x]$ stabilizes. But what can we say about this real number $x$? What is it? It is called $M_1^\sharp$ (read `$M$ one sharp'), and it's an instance of a major concept in modern set theory, that of a \emph{mouse}. In order to talk about mice, however, we first need to talk a little bit about large cardinals.

\subsection{What is a large cardinal?}
\label{sec:what-is-large}

\emph{Large cardinals} are a second major topic of modern set theory (in addition to inner models). Whereas inner models (and forcing) involve modifying models of set theory by changing their `width,' large cardinals have to do with the `height' (but also the width) of models of \ZFC{}. A large cardinal is a cardinal number $\kappa$ satisfying some strong infinitude property. For instance, a cardinal $\kappa$ satisfying that $V_\kappa$ is a model of \ZFC{} is called a \emph{worldly cardinal}. This is one of the weakest large cardinal properties, but it is somewhat representative in that most things that are referred to as large cardinal properties entail this.

One consequence of this is that it follows from a very basic incompleteness argument that \ZFC{} does not prove the existence of worldly cardinals. To see this, assume that there is a worldly cardinal $\kappa$. Without loss of generality, we may assume that $\kappa$ is the least worldly cardinal (which we can do since the cardinal numbers are well-ordered). Then $V_\kappa$ is a model of \ZFC{} in which there are no worldly cardinals.

This is an independence proof, but it has an important qualitative difference from something like the independence of the continuum hypothesis from \ZFC{}. With $\mathsf{CH}$, models that satisfy $\mathsf{CH}$ and models that satisfy $\neg \mathsf{CH}$ are able to `see each other.' For instance, given any model $M$ of \ZFC{}, $L$ computed inside $M$ will be a model of $\ZFC + \mathsf{CH}$. And, likewise, any model of \ZFC{} has a forcing extension which satisfies $\neg \mathsf{CH}$. This means that one can show in fairly weak metatheories that $\ZFC + \mathsf{CH}$ is a consistent theory (i.e., does not prove a contradiction) if and only if $\ZFC + \neg \mathsf{CH}$ is a consistent theory.

Large cardinals, by contrast, result in a `one-way' relationship. A model of $\ZFC + {}$``there exists a worldly cardinal'' is able to `see' a model of $\ZFC +{}$``there does not exist a worldly cardinal'', but the reverse is not true, as we discussed a moment ago. The first theory is fundamentally stronger than the second. It could be the case that $\ZFC+{}$``there exists a worldly cardinal'' is inconsistent while $\ZFC+{}$``there does not exist a worldly cardinal'' is not. This is analogous to G\"odel's second incompleteness theorem, but is arguably much more semantically natural. In particular, there is a wide variety of large cardinal axioms and they seem to fall in a natural hierarchy in which large cardinals of higher consistency strength are (often, but not always) profoundly larger in the sense of cardinality than large cardinals of lower consistency strength. For example, if there is a single \emph{inaccessible cardinal}, $\mu$, then the set $\{\kappa < \mu : \kappa~\text{is a worldly cardinal}\}$ has cardinality $\mu$.\footnote{This is an example of the \emph{reflection} phenomenon in set theory.}

Another reason for studying large cardinals is that they have a surprisingly strong relationship to descriptive set theory and the structure of $\Rb$, despite the fact that large cardinals are generally much, much larger objects than $\Rb$. Measurable cardinals were originally motivated by concepts from measure theory, but, arguably coincidentally, their existence also entails that $\Rb$ is more regular than it otherwise needs to be. Specifically, \ZFC{} can prove that $\Sigma^1_1$ sets are always Lebesgue measurable, but it cannot prove that $\Sigma^1_2$ sets are. By contrast, Solovay showed that $\ZFC+{}$``there exists a measurable cardinal'' \emph{does} prove that $\Sigma^1_2$ sets are always Lebesgue measurable \cite[Thm.~14.3]{Kanamori2003}, but Silver showed that it does not prove that every $\Sigma^1_3$ set is Lebesgue measurable \cite{Silver1971Meas}.

In \cref{prop:big-doesn't-work}, we assumed `enough' large cardinals to give an example of a real number $s$ which cannot be the value of $(\dagger)$. This real number $s$ has no nice ``snappy" characterization: it is achieved by interleaving the digits of a mysterious real $x$ and a generic real over $L[x]$. As it turns out, the real $x$, if it exists, also cannot be the value of $(\dagger)$. We believe it is possible to give an explanation of what this real number $x$ is, but unfortunately, the actual proof that it cannot be the value of $(\dagger)$ requires that the reader already have a background in descriptive inner model theory. Hence, we take the next section to attempt to explain what this real number $x$ is, and we give the proof that it cannot be the value of $(\dagger)$ in the Appendix.

\subsection{What is a mouse?}

\begin{figure}
  \centering
\includegraphics[width=0.9\textwidth]{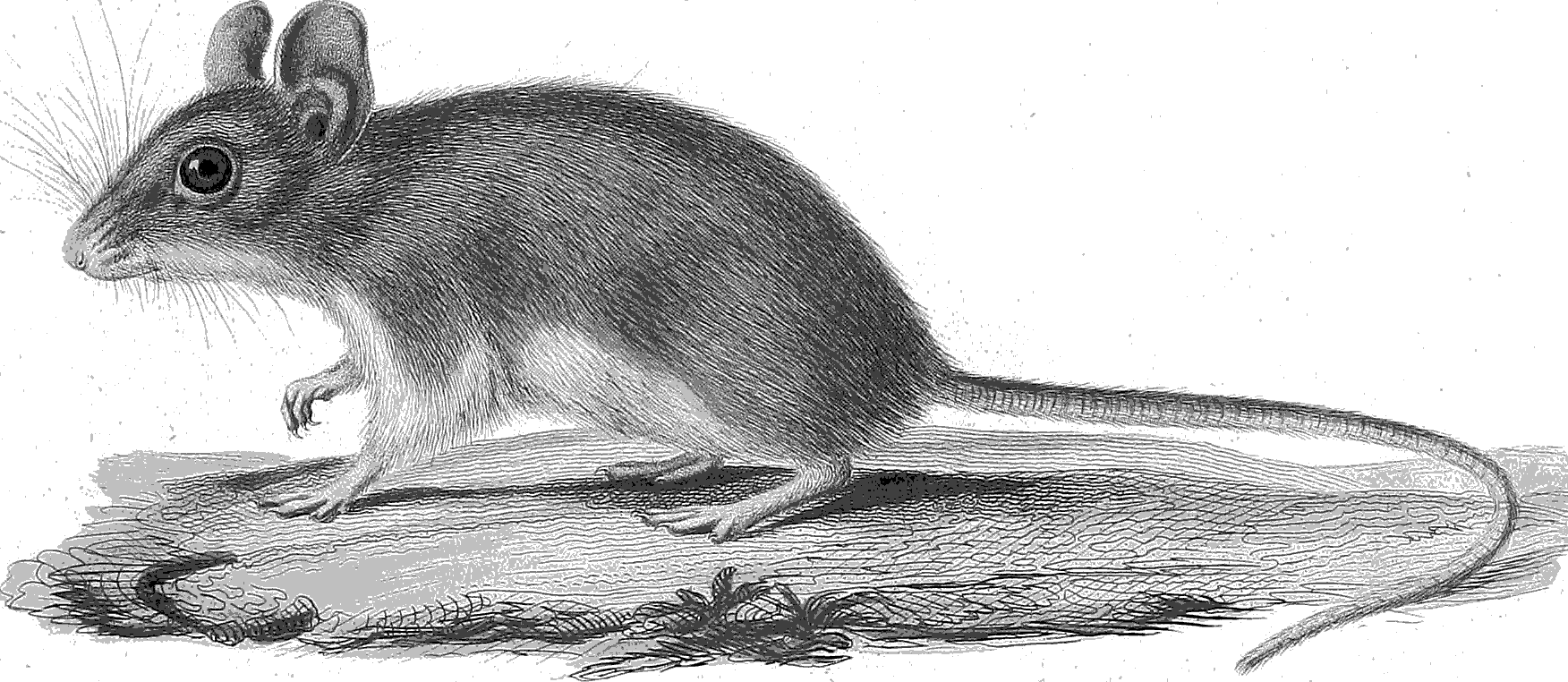}
\caption{A mouse \cite{bhlpage31059400}.}
\label{fig:mouse}
\end{figure}

Mice (see Figure~\ref{fig:mouse}) occur in the context of inner model theory, which we discussed briefly earlier in \cref{sec:encoding}. One of the primary motivations for inner model theory is the limitations of $L$ in the context of large cardinals. $L$ is incompatible with large cardinals of sufficiently high consistency strength. The first result in this direction was due to Dana Scott, who showed that $L$ is `too narrow' to accommodate measurable cardinals:






\begin{thm}[Scott \cite{SCOTT1961}]
  There cannot be a measurable cardinal in $L$.
\end{thm}

In fact, among the most commonly studied large cardinals, measurables represent a phase transition from large cardinals that can exist in $L$ to large cardinals that are too big to live in $L$: Inaccessible and weakly compact cardinals can exist in $L$, but measurable, strong, Woodin, and supercompact cardinals cannot.

There is a fundamental tension between the tightly structured nature of $L$ and $L$-like models of \ZFC{} and the richness entailed by the existence of large cardinals. In regards to the main topic of this paper, the reals in $L$ are `very thin' but in the presence of large cardinals, $\Rb$ is `very full' (and particular has mice in it). Inner model theory looks to resolve this tension by creating canonical models (also known as \textit{core models}) of \ZFC{} which enjoy some of the properties of $L$ which are useful in doing set theory, yet are large enough to accommodate large cardinals.

Given a large cardinal property, then, how do we build a canonical model in which the large cardinal property holds? The naive way to do it would be to take a witness $X$ that a cardinal $\kappa$ has a certain large cardinal property, and to consider the model $L[X]$. This ends up being how inner model theorists build the core model, but the way they build $L[X]$ is wildly different from \cref{defn:L-X}.\footnote{The reasons \cref{defn:L-X} is not sufficient, and exactly what kind of organization is needed to run the core model construction, are the subject of a subfield of set theory called \textit{fine structure theory}.} $L[X]$ ends up being constructed with a very complicated inductive process, and it is essentially the partial constructions of this inductive process which are called mice.

Once we move past measurability in the large cardinal hierarchy, measures become insufficient to capture that property, and we need to resort to what are called \textit{coherent sequences of extenders}. What these exactly are is beyond the scope of this paper, but roughly an extender is a generalization of a measure (in the sense of measure theory) which allows us to capture cardinals of higher consistency strength than measurable cardinals, and a so-called `coherent' sequence of them is necessary to capture roughly the large cardinal strength of a real number existing as in \cref{fact:big-enough}, namely a \textit{Woodin cardinal}. While the precise definition of a Woodin cardinal is beyond the scope of this paper, they are far stronger than measurable cardinals, and appear throughout descriptive set theory.

Thus, to build a core model with a Woodin cardinal, one must build a model of the form $L[\vec{E}]$, where $\vec{E}$ is a coherent sequence of extenders, but organized in a far more complicated way than we have done in \cref{defn:L-X} so that we may ensure that initial segments of this construction `play nicely' with one another.

With this, we make a pseudodefinition: 
\begin{pseudodefinition}
    A \textit{mouse} is a structure of the form\footnote{Strictly speaking mice are defined in terms of a modification of the $L$-hierarchy called the \emph{$J$-hierarchy}, but in the interest of simplicity, we are glossing over this.} $(L_\alpha[\vec{E}],\in,\vec{E})$ with $\vec{E}$ a coherent sequence of extenders,\footnote{In the structure $(L_\alpha[\vec{E}],\in,\vec{E})$, $\vec{E}$ is represented by a unary predicate satisfied by pairs $(\alpha,E_\alpha)$, where $E_\alpha$ is the extender in $\vec{E}$ with index $\alpha$.} which is an initial segment of the inductive core model construction.\footnote{There are other technical requirements on mice, such as \textit{soundness} and \textit{iterability}, which are far beyond the scope of this paper.} 
\end{pseudodefinition}

The mice may be well-ordered by the order in which they are built in the core model construction, which allows us to refer to the ``minimal mouse'' in which something happens.

\begin{definition}
    $M_1$ is the minimal mouse such that $M_1$ satisfies ``There exists a Woodin cardinal.''\footnote{An important thing to note with statements like this is that this doesn't mean $M_1$ literally contains a Woodin cardinal. It merely means that it contains an object that looks like a Woodin cardinal as far as $M_1$. This is similar to Skolem's paradox, the fact that countable models of \ZFC{} can contain sets that the model `thinks' are uncountable.}
\end{definition}

We may now form an object\footnote{We would like to remind the reader that doing this necessarily requires us to jump in consistency strength. Of course, it is consistent with \ZFC{} that $M_1$ exists but $M_1^\sharp$ does not, so when we say ``we may form an object $M_1^\sharp$, we assume that the necessary machinery exists for us to do so.} $M_1^\sharp$ which has the same relationship to $M_1$ as $0^\sharp$ does to $L$. The exact way in which they relate is beyond the scope of this paper.\footnote{For example, it is not equivalent to the existence of indiscernibles for $M_1$.} To state exactly the way in which this analogy holds would require us to recast $0^\sharp$ in terms of mice. With this in mind, we make a second pseudodefinition:

\begin{pseudodefinition}
    $M_1^\sharp$ is the minimal active\footnote{The definition of \textit{active} is to ensure that $M_1^\sharp$ has the same relationship to $M_1$ as $0^\sharp$ does to $L$. The precise meaning of this is beyond the scope of the paper.} mouse such that $M_1^\sharp$ satisfies ``There exist a Woodin cardinal.''
\end{pseudodefinition}

Officially, since $M_1^\sharp$ is a mouse, it is not a real number, but a structure. However, its theory may be coded by a real number---for instance, we could take the number whose $n$th bit is $1$ if and only if $M_1^\sharp$ satisfies the sentence coded by $n$---and this real number is the real number $x$ as in \cref{fact:big-enough} and \cref{prop:big-doesn't-work}.


\appendix

\section{$M_1^\sharp$ cannot be its value}

In \cref{prop:big-doesn't-work}, we assume enough large cardinals to give an example of a real number $s$ which cannot be the value of $(\dagger)$. This real number $s$ has no nice ``snappy" characterization: it is achieved by interleaving the digits of $M_1^\sharp$ and a generic real over $L[M_1^\sharp]$. In this section, we prove that $M_1^\sharp$ itself cannot be the value of $(\dagger)$.

\begin{theorem}
    Suppose $\mathsf{ZF+DC_\mathbb{R}+AD}$. Then for every real $x$ with $x\ge_T M_1^\sharp$, $\mathbb{R}\cap\mathsf{HOD}^{L[x]}=\mathbb{R}^{M_1}.$
\end{theorem}

\begin{proof}
    First, we show that $\mathbb{R}\cap\mathsf{HOD}^{L[x]}$ stabilizes on a cone at all. Let $\prec_x$ denote the canonical well-ordering of $\mathsf{HOD}^{L[x]}$ as computed in $L[x]$, and let $\prec_x^{\mathbb{R}}$ be its (degree-invariant) restriction to the reals of $\mathsf{HOD}^{L[x]}.$ Furthermore, let $r_x(\alpha)$ denote the $\alpha^{\text{th}}$ real of $\mathsf{HOD}^{L[x]}$, if it exists. For each $\alpha<\omega_1^V$, there is a cone $C_\alpha$ and a real $r(\alpha)$ such that for all $x\in C_\alpha,$ $r_x(\alpha)=r(\alpha)$. Fix such an $\alpha$ and such a real $r=r(\alpha)$ such that $$\exists b\forall x\ge_T b(r=r_x(\alpha)).$$ Denote this formula by $\varphi_\alpha(r).$

    Next, we show that actually $\mathbb{R}\cap\mathsf{HOD}^{L[x]}=\mathbb{R}^{M_1}$ on the cone above $M_1^\sharp$. For one direction of the inclusion, we wish to show that if $\alpha<\omega_1^V$ and $r\in\mathbb{R}$ satisfies $\varphi_\alpha(r)$, then $r\in\mathbb{R}^{M_1}.$ Let $b$ be a base witnessing $\varphi_\alpha(r)$. By Woodin's genericity iteration theorem, there is a normal iteration map $\pi: M_1^\sharp\to N$ where $N$ is an iterate of $M_1^\sharp$ such that for some $g\subseteq\mathrm{Coll}(\omega,\delta^N)$ which is $N$-generic, we have $b\in N[g].$ Work in $N[g]$. Consider the formula $\Phi_\alpha(y)$ expressing that $y$ is the unique real such that $\varphi_\alpha(y)$ holds. This is a first-order formula in $V$ with parameter $\alpha$ only. Because $\varphi_\alpha(r)$ is true in $V$, $\Phi_\alpha(r)$ is true in $V$. Because $N[g]$ is the sharp for a model of the form $L[V_\delta^N][g]$, we have that $N[g]\vDash\Phi_\alpha(y)"$, and that unique real number must be $r$. Hence, $r\in\mathrm{OD}^{N[g]}.$ Since the forcing $\mathrm{Coll}(\omega,\delta^N)$ is weakly homogenous over $N$, then since $r\in\mathrm{OD}^{N[g]}$, we have that $r\in N.$ Finally, note that $N\cap\mathbb{R}=M_1^\sharp\cap\mathbb{R}=M_1\cap\mathbb{R}$, so $r\in M_1.$
    
    For the other direction, we wish to show that if $x\in\mathbb{R}$ with $x\ge_T M_1^\sharp$, then $\mathbb{R}^{M_1}\subseteq\mathbb{R}\cap\mathsf{HOD}^{L[x]}$. Fix $x\geq_T M_1^\sharp.$ Let $\alpha<\omega_1^{M_1}$ be such that $M_1\upharpoonright\alpha$ projects to $\omega$. By \cite{Steel1995}, for such an $\alpha$, $M_1\upharpoonright\alpha$ is the unique premouse $N$ of height $\alpha$ which projects to $\omega$ and is $\Pi^1_2$-iterable. Let $\Psi(\alpha)$ be the statement $$``N\text{ is the unique premouse of height $\alpha$ which projects to }\omega\text{ and }N\text{ is }\Pi^1_2\text{-iterable}".$$ The statement $``N$ is $\Pi^1_2$-iterable" is uniformly $\Pi^1_2$ in the real coding $N$. By Shoenfield absoluteness, this statement is absolute between $V$ and $L[x]$. Since $x\ge_T M_1^\sharp$, $L[x]$ contains a code for $M_1^\sharp$, hence it contains codes for the relevant countable initial segments $M_1\upharpoonright\alpha$ with projectum $\omega.$ In particular, $L[x]$ contains some code witnessing $\Psi(\alpha)$, and Shoenfield absoluteness ensures $L[x]$ agrees with $V$ about $\Pi^1_2$-iterability of those codes. Therefore, inside $L[x]$, $M_1\upharpoonright\alpha$ is defined by the formula $\Psi(\alpha)$. Hence, $M_1\upharpoonright\alpha\in\mathsf{HOD}^{L[x]}$. A standard fine-structural fact states that the reals of $M_1$ appear exactly at those stages $\alpha$ where $M_1\upharpoonright\alpha$ projects to $\omega$. Thus,  every $y\in\mathbb{R}^{M_1}$ belongs to some $M_1\upharpoonright\alpha$ with projectum $\omega,$ and so $r\in\mathsf{HOD}^{L[x]}.$
\end{proof}
Thus, if it exists, $M_1^\sharp$ will not be a lightface projective singleton of $L[x]$ for any real $x$, as if it were, it would also be in $\mathsf{HOD}^{L[x]}$, and hence it cannot be the value of $(\dagger)$.

\bibliographystyle{vancouver}
\bibliography{ref}

\end{document}